\documentclass[11pt, reqno]{amsart}
\usepackage{hyperref}
\usepackage{cmap}                        % Ïîääåðæêà ïîèñêà ðóññêèõ ñëîâ â PDF (pdfla\TeX)
\usepackage[cp1251]{inputenc}            % Âûáîð ÿçûêà è êîäèðîâêè
\usepackage[english]{babel}
\usepackage[left=1in,right=1in,top=1in,bottom=1.5in]{geometry} % ïîëÿ ñòðàíèöû
\usepackage{amssymb,amsmath, amsthm, amscd,ifthen}
\usepackage{graphicx}
\usepackage{epigraph, xcolor, hyperref}

\newcommand{\N}{{\mathbb N}}
\newtheorem*{thm*}{Theorem}
\newcommand{\ff}{{\mathcal F}}

\newcommand{\G}{{\mathcal G}}
\newcommand{\h}{{\mathcal H}}

\newcommand{\D}{{\mathcal D}}

\newtheorem*{cla*}{Claim}

\newtheorem{thm}{Theorem}

\newtheorem{prb}{Problem}

\newtheorem{lem}[thm]{Lemma}
\newtheorem{cla}[thm]{Claim}
\newcommand{\eps}{{\varepsilon}}

\newtheorem{cor}[thm]{Corollary}
\date{}

\newtheorem{prop}[thm]{Proposition}

\newtheorem{defn}[thm]{Definition}

\DeclareMathOperator{\E}{\mathrm E}

\date{}
\title{Maximal degrees in subgraphs of Kneser graphs}
\author{Peter Frankl}\address{R\'enyi Institute, Budapest, Hungary and Moscow Institute of Physics and Technology, Russia; Email: {\tt peter.frankl@gmail.com}}

\author{Andrey Kupavskii}
\address{IAS Princeton, US and
Moscow Institute of Physics and Technology, Russia; Email: {\tt kupavskii@ya.ru}.} %\author{Peter Frankl}\address{R\'enyi Institute, Budapest, Hungary; Email: {\tt peter.frankl@gmail.com}}

\begin{document}
\maketitle
\begin{abstract} In this paper, we study the maximum degree in non-empty induced subgraphs of the Kneser graph $KG(n,k)$. One of the main results asserts that, for $k>k_0$ and $n>64k^2$, %and every positive $\eps$,
whenever a non-empty subgraph has $m\ge k{n-2\choose k-2}$ vertices, its maximum degree is at least $\frac 12(1-\frac {k^2}n) m - {n-2\choose k-2}\ge 0.49 m$. This bound is essentially best possible. One of the intermediate steps is to obtain structural results on non-empty subgraphs with small maximum degree.

%non-independent sets in Kneser graphs. In particular, we show that for any $k\ge k_0$ and such that if $n\ge 32k^{2}$ and a family $\mathcal F$ of $k$-element sets of $\{1,\ldots,n\}$ of size at least $C(\eps) k{n-2\choose k-2}$ is not intersecting, then there is a set $S\in \ff$ that is disjoint with almost a half of sets from $\ff$.
\end{abstract}

\section{Introduction}
Let $n\ge 2k>1$ be positive integers and let $[n]$ denote the standard $n$-element set $\{1,\ldots,n\}$. Set ${[n]\choose k}:=\{F\subset [n]: |F| = k\}$. Let $KG(n,k)$ denote the famous Kneser graph. Its vertex set is ${[n]\choose k}$, and two vertices $F,G\in {[n]\choose k}$ form an edge iff $F\cap G =  \emptyset$.

The interest in Kneser graphs goes back to 1955 when Kneser \cite{Knes} formulated the conjecture that the chromatic number $\chi(KG(n,k))$ equals $n-2k+2$. This conjecture was settled in an influential paper of Lov\'asz \cite{Lova} some twenty years later.

In the meantime, Erd\H os, Ko and Rado \cite{EKR} determined the independence number $\alpha(KG(n,k))$. Note that a family $\ff\subset {[n]\choose k}$ is an independent set in $KG(n,k)$ iff $F\cap G\ne \emptyset $ for all $F,G\in {[n]\choose k}$. In extremal set theory such an $\ff$ is called {\it intersecting}.

\begin{thm}[\cite{EKR}] For $n\ge 2k>0$, \begin{equation}\label{eqekr}
                                         \alpha(KG(n,k)) = {n-1\choose k-1}.
                                       \end{equation}
\end{thm}
Equality in \eqref{eqekr} is attained for the star $\mathcal S_x:=\{S\in {[n]\choose k}: x\in S\}$. Hilton and Milner \cite{HM} proved that for $n>2k$ no other independent set (intersecting family) attains the size ${n-1\choose k-1}$. More specifically, they showed that if no element is contained in all sets of an intersecting family, then it is at most as large as $\mathcal H(x,F)$, where for $x\notin F\in {[n]\choose k}$ we define $$\mathcal H(x,F):=\{F\}\cup \{A\in{[n]\choose k}: x\in A, A\cap F\ne \emptyset\}.$$

\begin{defn}
  For a family $\ff\subset {[n]\choose k}$, let $KG(\ff)$ denote the induced subgraph of $KG(n,k)$ on the vertex set $\ff$. Let $e(\ff)$ {\em ($d(\ff)$)} denote the number of edges (maximum degree) of $KG(\ff)$.
\end{defn}

In view of \eqref{eqekr}, for $|\ff|>{n-1\choose k-1}$ both $e(\ff)$ and $d(\ff)$ are positive. Defining $\mathcal S^+:= \mathcal S_x\cup \{T\},$ where $x\notin T\in {[n]\choose k},$ one easily verifies $$e(\mathcal S^+) = d(\mathcal S^+) = {n-k-1\choose k-1}.$$

Katona, Katona, and Katona \cite{KKK} proved
\begin{equation}\label{eqkat}
  e(\ff)\ge {n-k-1\choose k-1}
\end{equation}
for all $\ff\subset {[n]\choose k}$ with $|\ff| = {n-1\choose k-1}+1$. This result was extended to the case
$|\ff| \le  {n-1\choose k-1}+\frac {n-2k}n{n-k-1\choose k-1}$ by Balogh et. al. \cite{Bal} (see Theorem~\ref{thmde} below).

Very recently, Friedgut (personal communication) raised the problem of determining the minimum of $d(\ff)$ for $|\ff| = {n-1\choose k-1}+1$. One of the motivations for this question is a recent breakthrough result of Huang \cite{Hua}, who showed that in any subset of the hypercube $\{0,1\}^n$ of size $2^{n-1}+1$ the maximum degree of a vertex (in the standard hypercube graph) is at least $\sqrt n$. This settled an old problem from Theoretical Computer Science, known as the Sensitivity Conjecture.

Getting back to Kneser graphs, let us state this problem in a more general form.

\begin{prb}
Define $$d(m,n,k):= \min \Big\{d(\ff): \ff\subset {[n]\choose k}, |\ff| = m \text{ and } \ff \text{ is not intersecting}\Big\}.$$
Determine or estimate $d(m,n,k)$.
\end{prb}
We should note that $KG(2k,k)$ is a perfect matching. Thus $d(m,2k,k) = 1$ identically for $2\le m\le {2k\choose k}$.

{\bf Example. } Consider the family $\D:= \mathcal H(x,F)\cup \{F'\}$, where $F'$ contains $x$ and is disjoint from $F$. Obviously, $|\D| = {n-1\choose k-1}-{n-k-1\choose k-1}+2$, $e(\D) = d(\D) = 1$. This example shows that the problem of estimating $d(m,n,k)$ is only interesting for  $m>{n-1\choose k-1}-{n-k-1\choose k-1}+2$. Note that, for $k\ge 4$, we have $|\D|\ge k{n-k\choose k-2}$.

\begin{defn}
  The {\it lexicographic order} $A<_L B$ is defined for distinct $A,B\in {[n]\choose k}$ by $A<_L B$ iff $\min \{x:x\in A\setminus B\}<\min \{y: y\in B\setminus A\}$. let $\mathcal L(m):= \mathcal L(m,[n],k)$ denote the family of the first $m$ members in ${[n]\choose k}$ in the lexicographic order. Note that if $1\le \ell \le n-k$ is an integer then $\mathcal L({n\choose k}-{n-\ell\choose k}) = \{A\in {[n]\choose k}: A\cap [\ell]  \ne \emptyset\}$.
\end{defn}
For a wide range of the values of the parameters Das, Gan, and Sudakov \cite{DGS} proved that $e(\ff)\ge e(\mathcal L(|\ff|)$. Later, their results were extended to another range by Balogh et. al. \cite{Bal}.
\begin{thm}[\cite{DGS}, \cite{Bal}]\label{thmde}
  Suppose that $m,n,k,\ell$ are positive integers, $m\le {n\choose k}-{m-\ell \choose k}$ and $n\ge 108 \ell k^2(k+\ell)$ {\em (\cite{DGS})} or $n\ge Ck^2\ell^3$ with some absolute $C$ {\em (\cite{Bal})}. Then
  \begin{equation}\label{eqdisj}
    e(\ff)\ge e(\mathcal L(m)) \ \ \ \ \ \text{for all } \ff\subset {[n]\choose k}\text{ satisfying } |\ff| = m.
  \end{equation}
The same holds for any $n>2k$ and $|\ff| \le  {n-1\choose k-1}+\frac {n-2k}n{n-k-1\choose k-1}$ {\em (\cite{Bal})}.
\end{thm}
More generally, the number of pairs of sets with intersection at most $t$ was studied in \cite{FKR}.\\

Note the obvious relationship
\begin{equation*}\label{eqdegedge}
  d(\ff)\ge 2 e(\ff)/|\ff|.
\end{equation*}

For $1\le i<k$ and a set $P\in {[n]\choose i},$ we use the standard notation $\ff(P):=\{F\setminus P: P\subset F\in \ff\}$ and $\ff(\bar P):=\{F\in \ff: F\cap P = \emptyset\}$.

With this notation, $d(\ff) = \max\{|\ff(\bar F)|: F\in\ff\}$. Recall that, for a family $\ff$, its {\it covering  number} $\tau(\ff)$ is the minimum size of $S\subset [n]$ such that $\ff(\bar S) = \emptyset$.\\

The structure of the extremal examples that minimize $d(\ff)$, although related to the families $\mathcal L(m)$, appears to be significantly more complicated.
Our first two results are structural and apply in a more general setting.

\begin{thm}\label{thmnew1} Fix an integer $t\ge 2$ and let $k_0$ be sufficiently large.  If $k\ge k_0$ and $n\ge 16t^2k^2$ then the following holds.  Let $\ff\subset {[n]\choose k}$ %be non-intersecting and
satisfy $|\ff|\ge |\D|$. % {n\choose k}-{n-2\choose k}$.
  Then either $d(\ff)\ge \big(1-\frac 1t\big)|\ff|$ or %, for any family $\ff$ that minimizes $d(\ff)$ for a fixed $|\ff|$,
there exists a set $S$ of size at most  $t$ such that $|\ff(\bar S)|\le {n-4\choose k-4}$. Moreover, for each $x$ in $S$, we have %if $|S| = 2$ then
$|\ff(x)|\ge k^{-1/2}|\ff|$.% for both $x\in S$. % must have covering number $2$.
\end{thm}
{\bf Remark. } It is sufficient to take $k_0$ such that $(10t\log k_0)^7\le k_0$.
Here and in what follows, all $\log$'s are base $2$.\\

%We note that ${n-4\choose k-4}$ above may be replaced by ${n-c\choose k-c}$ for any fixed constant $c$, provided that $32$ is replaced by a constant $C = C(c)$ and $k\ge k_0(c)$.
For somewhat large $\ff$, we can get rid of $\ff(\bar S)$ completely.
\begin{thm}
Fix and integer $t\ge 2$ and let $k_0$ be sufficiently large.  Let further $k\ge k_0$ and $n\ge 16t^2k^2$.  If  $\ff\subset {[n]\choose k}$ and $|\ff|\ge 2t^2k{n-2\choose k-2}$, then $d(\ff)\ge  \big(1-\frac 1t\big)|\ff|$ unless $\ff(\bar S) = \emptyset$, where $S$ is as in Theorem~\ref{thmnew1}. In particular, if $d(\ff)< \big(1-\frac 1t\big)|\ff|$ then $\tau(\ff)\le t$.
\end{thm}
\begin{proof}
  Assume that $d(\ff)<\big(1-\frac 1t\big) |\ff|$. In view of Theorem~~\ref{thmnew1}, we can choose a set $S$, $|S|\le t$, satisfying $|\ff(\bar S)|\le {n-4\choose k-4}$. If $\ff(\bar S) = \emptyset$ then we are done. Otherwise fixe some $U\in \ff(\bar S)$. Consider an arbitrary set $F\in \ff\setminus \ff(\bar U)$. Then either $F\in \ff(\bar S),$ or $F\cap U\ne \emptyset$ and $F\cap S\ne \emptyset$ hold simultaneously. Thus $|\ff|-d(\ff)\le |\ff\setminus \ff(\bar U)|\le {n-4\choose k-4}+tk{n-2\choose k-2}<\frac 1t |\ff|$, a contradiction. Therefore, $|\ff(\bar S)| = \emptyset $ must hold.
  %Next, assume that $\ff(\bar S)$ is not empty and take a set $U\in \ff(\bar S)$. There are at most $|S|\cdot k{n-2\choose k-2}\le tk{n-2\choose k-2}$ sets from $\ff\setminus \ff(\bar S)$, along with ${n-4\choose k-4}$ sets from $\ff(\bar S)$, that intersect $U$. This is less than $\frac 1t|\ff|$ in total.
\end{proof}
The next theorem focuses on the case $t=2$ and the application to the maximum degree question for non-intersecting families. It gives a more precise structural information depending on the size of $\ff$. The situation turns out to be quite complicated for small values of $|\ff|$. Consider the following examples:
$$\mathcal E_i: = \Big\{F\in {[n]\choose k}: 1\in F, F\cap [2,k+i+1]\ne \emptyset\Big\}\cup {[2,k+i+1]\choose k}.$$
Note that, for $i<k$ and $n>k^2\ge 100$,
\begin{equation}\label{eqdei}d(\mathcal E_i) = |\mathcal E_i([2,k+1])| = {n-k-1\choose k-1}-{n-k-i-1\choose k-1}.
\end{equation}
\begin{thm}\label{thmnew2} Suppose that $k\ge k_0$ and $n\ge 64k^2$. Let $\ff\subset {[n]\choose k}$ satisfy $d(\ff) = d(|\ff|,n,k).$ Then there exists $S$ as in Theorem~\ref{thmnew1} such that the  following holds.
\begin{enumerate}
  \item If, for some integer $1\le i\le \frac k2$, \begin{equation}\label{eqcondnew2} |\D|\le |\ff|\le {n-1\choose k-1}-{n-k-i-1\choose k-1}+{k+i\choose k},\end{equation} then $S = \{x\}$ for some $x\in [n]$ and $\ff(\bar x)\subset {Y\choose k}$ for some set $Y$ of size $k+i$, $x\notin Y$. In particular, $|\ff(\bar x)|\le {k+i\choose k}$. Moreover, if equality holds in \eqref{eqcondnew2} then $\ff$ is isomorphic to $\mathcal E_i$.
  \item If $|\ff|\ge 2|\D|$ then $|S| = 2$.
  \item If $|\ff|\ge 4|\D|$ then $\ff(\bar S) = \emptyset$.
\end{enumerate}
\end{thm}
Case 2 of the theorem is  necessary since, while the example for the tightness of case 3 is expected to be of the following type:
$$\mathcal W_\ell:= \Big\{F\in {[n]\choose k}: \{1,2\}\subset F\Big\}\cup \Big\{F\in {[n]\choose k}:|F\cap[1,2]| =1, F\cap [3,\ell+2]\ne \emptyset\Big\}$$
(or, more generally, $\mathcal L(m)$-type on $[3,n]$), for small $\ell$ these examples may be `improved' by adding the following family:
$$\mathcal W'_{\ell'}:= \Big\{F\in {[n]\choose k}: |F\cap [3,\ell+2]|\ge \ell'\Big\}.$$
Indeed, if $\ell'$, say, satisfies $\ell'>\frac 23\ell>\frac {3k}4$, then $\mathcal W_\ell\cup \mathcal W'_{\ell'}$ beats families with $|S| = 1$ and has the same maximum degree as $\mathcal W_\ell$ since the elements from $\mathcal W'_{\ell'}$ have low degree.\\

Our next result gives a concrete numerical implication of Theorems~\ref{thmnew1}.
\begin{thm}\label{thm3} In the conditions of Theorem~\ref{thmnew2} (3), we have \begin{equation}\label{eq10}d(\ff)\ge \frac{\frac{|\ff|-{n-2\choose k-2}}2{n-k-1\choose k-1}}{{n-2\choose k-1}} - {n-k-2\choose k-2}\ge \frac 12\Big(1-\frac{k^2}n\Big)|\ff|-\frac 32{n-k-2\choose k-2}.\end{equation}
\end{thm}
\vskip+0.1cm
{\bf Tightness.}
Although the bound \eqref{eq10} may appear somewhat arbitrary, it is actually optimal up to some lower order terms. See Section~\ref{secopt} for details.

\section{Simple bounds}
In this section, we present proofs of several weaker bounds which are, however, less technical. The aim is to convey a feeling of the problem to the reader. Moreover, the methods used here are different and may be interesting in their own right.

Throughout this section, we fix a family $\ff\subset {[n]\choose k},$ $|\ff|>{n-1\choose k-1}$. Note that, due to the Erd\H os--Ko--Rado theorem, $\ff$ is not intersecting. Our goal is to present several relatively simple cases in which $d(\ff)>0.49|\ff|$ and approaches to prove such a bound.

 For $1\le i<k$, define $c(i) = c(i,\ff) = {n-i\choose k-i}^{-1}\max\big\{|\ff(P)|: P\in {[n]\choose i}\big\}.$ %In particular, $c(0) = |\ff|$.
%\begin{lem} We have \begin{equation}\label{eq4} c(1)\le c(2)\le\ldots\le c(k) = 1. \end{equation}
%\end{lem}
%\begin{proof}
%  Fix $1\le i<k$ and $P\in {[n]\choose i}$ satisfying $|\ff(P)| = c(i){n-i\choose k-i}$. Note that $\ff(P)\subset {[n]\setminus P\choose k-i}$. Thus the average degree of $\ff(P)$ is $\frac {k-i}{n-i}|\ff(P)| = c(i){n-i-1\choose k-i-1}$. Consequently, we may choose $x\in [n]\setminus P$ satisfying
%  $$|\{G\in \ff(P):x\in G\}|\ge c(i){n-i-1\choose k-i-1}.$$
%  since the LHS is exactly $|\ff(P\cup \{x\})|,$ $c(i)\le c(i+1)$ follows. As $c(k)=1$ is evident, the proof is complete.
%\end{proof}

\begin{lem}\label{lemsimple1}
  Define $\gamma:=|\ff|/{n-1\choose k-1}$. Then
  \begin{equation}\label{eq55}
    d(\ff)\ge \frac 12\Big(1-\frac {c(2)k^3}{\gamma n}\Big)|\ff|.
  \end{equation}
\end{lem}
\begin{proof}
  Since $\ff$ is not intersecting, we may choose $G,H\in \ff$, $G\cap H = \emptyset$. If $F\in \ff$ intersects both $G$ and $H$ then it contains at least one of the $k^2$ pairs $(x,y)$ with $x\in G, y\in H$. This allows at most $k^2c(2) {n-2\choose k-2}$ such sets. Note that
  $${n-2\choose k-2}<\frac kn{n-1\choose k-1} =
 \frac{k}{\gamma n}|\ff|.$$
Now \eqref{eq55} follows from the fact that each of the remaining $F\in \ff$ are disjoint to either $G$ or $H$ (or both).
\end{proof}

This result immediately implies the following corollary.

\begin{cor}
  If $n\ge 50 k^3$ then $d(\ff)>0.49|\ff|$.
\end{cor}
\begin{proof}
   Then we can apply Lemma~\ref{lemsimple1} and use $\gamma>1$, $c(2)\le 1$.
\end{proof}

For $k\ge 50$ we can prove something stronger.
\begin{thm}\label{thmk50}
  If $n\ge 4k^3$ and $k\ge 50$ then $d(\ff)>0.49|\ff|$.
\end{thm}

\begin{proof}
For two families $\G\subset {[n]\choose g},$ $\h\subset {[n]\choose h}$, let $e(\G,\h)$ denote the number of disjoint pairs $G,H$, where $G\in \G$, $H\in \h$. In the case $\h = \{H\}$ consists of one member, we define $e(\G, H):=e(\G,\{H\})$.

The following simple lemma will be essential for the proof.

\begin{lem}
  Suppose that $1\le i<k$, $P\in {[n]\choose i}$ and $|\ff(P)| = c(i) {n-i\choose k-i}$. then
  \begin{align}\label{eq6}
    e(\ff(P),H)\ge  & |\ff(P)|-c(i+1)k{n-i-1\choose k-i-1}\ \ \ \text{for all }H\in \ff(\bar P), \\
   \label{eq7} e(\ff(P),\ff(\bar P))\ge & \Big(1-\frac {c(i+1)k^2}{c(i)n}\Big)|\ff(P)||\ff(\bar P)|.
  \end{align}
\end{lem}
\begin{proof}
  Fixing $H\in \ff(\bar P)$, $G\cap H\ne \emptyset$ for some $G\in \ff( P)$ is equivalent to $G\cup P\in \ff$ containing at least one of the $k$ sets $P\cup \{x\}$, $x\in H$. This allows for less than $c(i+1)k{n-i-1\choose k-i-1}$ such sets, implying \eqref{eq6}.

  To deduce \eqref{eq7}, simply sum \eqref{eq6} over all $H\in \ff(\bar P)$ and use $|\ff(P)| = c(i){n-i\choose k-i}$ together with ${n-i-1\choose k-i-1} = \frac{k-i}{n-i}{n-i\choose k-i} < \frac kn{n-i\choose k-i}.$
\end{proof}

\begin{cor}
  \begin{equation}\label{eq8}
    d(\ff)\ge \max\Big\{\frac 12, 1-c(1)\Big\}\cdot \Big(1-\frac {c(2)k^2}{c(1)n}\Big) |\ff|.
  \end{equation}
\end{cor}
\begin{proof}
  Note $|\ff|  = |\ff(x)|+|\ff(\bar x)|$. This implies $\max\{|\ff(x)|,|\ff(\bar x|)\}\ge \frac 12|\ff|$. Moreover, $|\ff(\bar x)|\ge |\ff|-c(1){n-1\choose k-1}> (1-c(1))|\ff|$. Applying \eqref{eq7} with $i=1$ and averaging yields \eqref{eq8}.
\end{proof}

\begin{lem}
If $d(\ff)< \frac 12|\ff|$ then we have  \begin{equation}\label{eq3}
  c(1)>\frac {|\ff|}{2k{n-1\choose k-1}}>\frac 1{2k}.
\end{equation}
\end{lem}

\begin{proof}
  Let $F\in \ff$ be arbitrary. In view of our assumption, $|\ff(\bar F)|<\frac 12|\ff|$. Equivalently, $|\ff\setminus \ff(\bar F)|>\frac 12 |\ff|.$ On the other hand, $|\ff\setminus \ff(\bar F)|\le \sum_{x\in F}|\ff(x)|\le kc(1){n-1\choose k-1}$. Comparing these two inequalities, the claim follows.
\end{proof}

%Note that $|\ff(\bar x)|>(1-c(1)){n-1\choose k-1}$ for any $x$, and so if $c(1)<\frac 12$ then it is better to use the following bound, obtained in the same way.

Let us continue with the proof of Theorem~\ref{thmk50}. In view of the lemma above, we may assume that $c(1)>\frac 1{2k}$. If $c(1)\ge \frac 12$ then \eqref{eq8} combined with $n\ge 100 k^2$ implies $d(\ff)\ge 0.49|\ff|$. If $\frac 1{2k}\le c(1)<\frac 12$ then the right hand side of \eqref{eq8} is at least $(1-c(1))\big(1-\frac 1{4c(1)k}\big)|\ff|$ due to $n\ge 4k^3$. It is an easy calculation that the minimum of this expression is attained for $c(1) = \frac 1{2k}$ or $\frac 12$ and the expression is at least $ 0.49|\ff|$ in both cases.
\end{proof}

Next, we add one more idea and prove the same statement in a yet wider range.

\begin{thm}
  If $n\ge 100k^2$ then $d(\ff)\ge 0.49|\ff|$.
\end{thm}
\iffalse
\begin{proof} First, note that, in view of Lemma~\ref{lemsimple1}, we may assume that $$c(2)\ge \frac 1{50}.$$ Using \eqref{eq8} as above, we are done for $c(1)\ge \frac 12$ and we can get the bound $d(\ff)\ge (1-c(1))(1-\frac 1{c(1)k})|\ff|$ for $c(1)<\frac 12$. In particular, if $c(1)\ge \frac 1{20}$ then $d(\ff)\ge (1-c(1))(1-\frac 1{50c(1)})|\ff|$. As a function of $c(1)$, this expression is minimized for either $c(1) = \frac 12$ or $\frac 1{20}$. In both cases, we get a value bigger than $0.49|\ff|.$ Thus, we can assume that $$c(1)\le \frac 1{20}.$$
Next, using that $c(8)\le 1$, we have $$c(2)\cdot \frac {c(3)}{c(2)}\cdot \frac {c(4)}{c(3)}\cdot \frac {c(5)}{c(4)}\cdot \frac {c(6)}{c(5)}\cdot \frac {c(7)}{c(6)}\cdot \frac {c(8)}{c(7)}\le 1,$$
and, given that $c(2)\ge \frac 1{50}$, there is $i\in [2,7]$ such that $\frac {c(i+1)}{c(i)}\le 2$. Take $P$, $|P| = i$, such that $|\ff(P)| = c(i){n-i\choose k-i}$. Note that, given that $i\le 7$, we have $|\ff(\bar P)|\ge |\ff|-7c(1){n-1\choose k-1}\ge 0.6|\ff|$. Using \eqref{eq7} and averaging, we get that
$$d(\ff)\ge \Big(1-\frac{2k^2}{n}\Big)|\ff(\bar P)|>\frac 12 |\ff|.$$
\end{proof}

We remark that $n$ being quadratic in $k$ is sufficient in the last case of the proof.
\fi
\begin{proof}
Assume that $c(2)/c(1) \le 5$. If $c(1)\le \frac 12$ then \eqref{eq8} gives $d(\ff)\ge (1-c(1))(1-\frac {c(2)k^2}{c(1)n})|\ff|\ge (1-c(1))(1-\frac {c(2)}{100c(1)})|\ff|$. Given the assumption, it is not difficult to see that this expression is minimized for $c(1) = \frac 12$, in which case we have $d(\ff)\ge \frac 12(1-\frac {c(2)}{50})|\ff| \ge 0.49|\ff|$. % for $n\ge 100is minimized >\frac 12|\ff|$. If $\frac 1
If $c(1)\ge \frac 12$ then $\frac {c(2)}{c(1)}\le 2$ and \eqref{eq8} implies $d(\ff)\ge \frac 12(1-\frac {2k^2}n)|\ff|\ge 0.49|\ff|$.

Assume that $c(2)/c(1)\ge 5$. Define $q$ to be the smallest integer satisfying $2^q\ge \frac 1 {c(2)}$. Since $c(i)\ge 1$  for all $1\le i\le k$ we have
$$\frac {c(3)}{c(2)}\cdot \frac {c(4)}{c(3)}\cdot\ldots\cdot \frac {c(q+2)}{c(q+1)}\le \frac 1{c(2)}\le 2^q.$$
Consequently we can fix $2\le i\le q+1$ such that
$$\frac {c(i+1)}{c(i)}\le 2.$$
By definition, $c(2)<2^{-(q-1)}$. Noting that $2^{q}\ge q+1$ for all $q\ge 0$, $ic(2)\le (q+1)c(2)<(q+1)2^{-(q-1)}\le 2$ follows.

%Then, using that $c(i)\le 1$ for any $i$, we have
%$$c(2)\cdot \frac {c(3)}{c(2)}\cdot \frac {c(4)}{c(3)}\cdot\ldots\cdot \frac {c(j+1)}{c(j)}\le 1,$$
%where $j = 1+\lceil\log \frac 1{c(2)}\rceil$.  Due to the choice of $j$, there must exist $2\le i\le j$, such that $\frac {c(i+1)}{c(i)}\le 2$.
Take $P$, $|P| = i$, such that $|\ff(P)| = c(i){n-i\choose k-i}$. Note that $i\le j = 1+\lceil\log_2 \frac 1{c(2)}\rceil\le \frac 2{c(2)}$, and so we have $|\ff(\bar P)|\ge |\ff|-ic(1){n-1\choose k-1}\ge |\ff|-\frac{ic(2)}5{n-1\choose k-1}\ge |\ff|-\frac 25{n-1\choose k-1}\ge 0.6|\ff|$.
Using \eqref{eq7} and averaging, we get that
$$d(\ff)\ge \Big(1-\frac{2k^2}{n}\Big)|\ff(\bar P)|\ge 0.98\cdot 0.6|\ff|>\frac 12 |\ff|.$$
\end{proof}

\section{Proof of Theorem~\ref{thmnew1}}
%Recall the definition of family $\D$ from the introduction.

The proof is a combination of Lemmas~\ref{lemnew1}, \ref{lemnew2}, and~\ref{lemnew3}.

\begin{lem}\label{lemnew1}
For any $c\in \N$ there exists $k_0$ such that for any $k\ge k_0, n\ge k^2+k$ and a family $\G\subset {[n]\choose k}$ of size at least ${n-c\choose k-c}$, we can find a collection $\G'\subset \G$ of $k^c$ sets such that $|G_1\cap G_2|\le \log k$ for any $G_1,G_2\in \G'$.
\end{lem}
\begin{proof}
  The proof is a simple probabilistic argument. Form a family $\G'\subset \G$ by including each set from $\G$ into $\G'$ independently with probability $p:=\frac {2k^c}{|\G|}$. It is an easy calculation to see that the number of sets in $\G'$ is at least $\frac {k^c}{|\G|}$ with probability at least $1/2$.

  Now let us calculate the expected number of pairs $A,B\in \G'$ that intersect in at least $1+\log k$ elements. The number of such pairs in $\G$ is at most $|\G|\cdot D$, where {\small $$D:= \sum_{i=1+\log k}^k{k\choose i}{n-k\choose k-i}\le {k\choose \log k}{n-k\choose k-\log k}\le \frac {k^{\log k}}{(\log k)!} \cdot \frac {k^{\log k-c}}{(n-k)^{\log k-c}}{n-c\choose k-c}\le \frac {k^c}{(\log k)!}{n-c\choose k-c}.$$} Remark that $(\log k)!\gg k^{C}$ for any constant $C$, provided $k$ is large enough. Thus, the expected number of pairs of pairs $A,B$ as above is at most
  $$p^2|\G|D\le \frac{4k^{3c}}{(\log k)!}<\frac 12,$$
  provided $k$ is large enough. Thus, by Markov's inequality, with probability strictly greater than $1/2$, there are no such pairs in $\G'$. Fix such a choice of $\G'$ that satisfies both properties simultaneously.
\end{proof}
Using a bit of algebra, we can get a stronger statement. We do not require it for the proof since it would improve the bounds on $k$ only slightly, but decided to keep it for the interested reader.
\begin{lem}\label{lemnew1v2}
For any $c\in \N$ there exists $k_0$ such that for any $k\ge k_0,\ n\ge 2k^2$ and a family $\G\subset {[n]\choose k}$ of size at least ${n-c\choose k-c}$, we can find a collection $\G'\subset \G$ of $k^c$ sets such that $|G_1\cap G_2|\le 4c$ for any $G_1,G_2\in \G'$.
\end{lem}
\begin{proof}
  Fix integer $d>0$. Take the largest prime power $q$ that satisfies $k\le q<\frac nk$. %For simplicity, let $k$ be prime and
  Consider $GF(q)$ and take some set $U\subset GF(q)$ of size $k$. Consider the graphs of all polynomials of degree at most $d$ in $U\times GF(q)$:
  $$G(f):=\{(x,f(x)): x\in U\}\subset U\times GF(q).$$
  For $f'\ne f$, $|G(f)\cap G(f')|\le d$ is obvious. Consider the family $\G(d)$ of all such $G(f)$. Then $|\G(d)| = q^{d+1}$.
  Next, consider a random injection $\phi: U\times GF(q)\to [n]$. This defines the family $\phi(\G(d)):= \{\phi (G): G\in \G(d)\}.$ By linearity of expectation,
  $$\E[|\phi(\G(d))\cap \G|] = \frac{|\G|q^{d+1}}{{n\choose k}} =:y.$$
  Note that $y\sim \frac {k^cq^{d+1}}{n^c}$ and that $q\sim \frac nk$ for large $k$. Thus, clearly $y\ge k^c$ for $d\ge 4c$. We get that there is a choice of $\phi$ that gives at least $k^c$ images of sets from $\G(d)$ in $\G$. Then $\G':=\phi(\G(d))$ is the desired subfamily.
\end{proof}

We need the following easy claim for the next lemma.
%\begin{cla}\label{clanew1} Let $n\ge k\ell$, where $n, k,\ell$ are integers. Then, for any given set $G\in {[n]\choose k}$, at most $\frac {k\ell }n$-proportion  of sets from ${[n]\choose \ell}$ intersect $G$.
%\end{cla}
%\begin{proof}
%  The proof is a simple calculation. The number of sets intersecting $G$ is ${n-1\choose \ell-1}+\ldots+{n-k\choose \ell-1}\le k{n-1\choose \ell-1} = \frac {k\ell}n{n\choose \ell}$.
  %, and we have $${n\choose k}/{n-k\choose k} = \prod_{i=0}^{k-1}\frac {n-i}{n-k-i}\le \Big(1+\frac {k}{n-2k}\Big)^k\le e^{\frac {k^2}{n-2k}}\le \min\Big\{4,1+\frac {2k^2}{n}\Big\},$$
%provided $k\ge 10$.
%\end{proof}
\begin{cla}\label{clanew2}
  Fix positive integers $n,k,s$ and let $n\ge k^2$. Assume that $F_1,\ldots, F_s$ are pairwise disjoint sets of size at most $k$. Then the proportion of $k$-element sets intersecting each of $F_1,\ldots, F_s$ is at most $\big(\frac {k^2}n\big)^s$.
\end{cla}
\begin{proof}
  There are at most $k^s$ transversal sets $T$ with $|T| = s$, $|T\cap F_i| = 1$ for all $i$. Each transversal $T$ is contained in ${n-s\choose k-s}<{n\choose k}\cdot \Big(\frac kn\Big)^s$ sets from ${[n]\choose k}$.
\end{proof}

%\begin{proof}
%  The proof is by induction on $s$. For $s=1$ this is given by Claim~\ref{clanew1}. Now assume that we have a family $\mathcal W\subset {[n]\choose k}$ of all sets that intersect $F_1,\ldots, F_{s-1}$. Fix a particular intersection pattern $P'\subset F_1\cup\ldots F_{s-1}$ with these sets and consider a subfamily $\mathcal W_{P}:= \{A\setminus P: A\in \mathcal W, A\cap (F_1\cup \ldots\cup F_{s-1}) = P\}.$ Then, clearly, the uniformity of all sets in $\mathcal W_{P}$ is the same and it is $w\le k-s+1$. Moreover,  $X:=[n]\setminus(F_1\cup\ldots F_{s-1}$ satisfies $|X| \ge n-(s-1)k$. Applying Claim~\ref{clanew1} to $\mathcal W_{P}$, we get that at most a $\frac {w |F_s|}{|X|}\le \frac {(k-s+1)k}{n-(s-1)k}\le \frac {k^2}n$ proportion of sets from  $\mathcal W_{P}$ intersect $F_s$. Since this is true for any choice of $P$, the number of sets intersecting $F_1,\ldots, F_s$ is at most a $\frac {k^2}n$-proportion of those intersecting $F_1,\ldots, F_{s-1}$, and the claim follows.
%\end{proof}
\begin{lem}\label{lemnew2}
  In the requirements of Theorem~\ref{thmnew1}, assume that there exists $S$ such that $|\ff(\bar S)|\ge {n-4\choose k-4}$ and, moreover, for any $i\in [n]\setminus S$ we have $|\ff(i)|\le \frac {|\ff|}{100(t\log k)^3}$. Then there is a set $F\in \ff(\bar S)$ that is disjoint with more than $\frac {(t-1)|\ff|}t$ sets from $\ff$.
\end{lem}
\begin{proof}
  Apply Lemma~\ref{lemnew1} and get a family $\G'\subset \ff(\bar S)$ of $10 t\log k$ sets with pairwise intersections of size at most $\log k$. Then the set $I = \bigcup_{A\ne B\in \G'} A\cap B$ has size at most $50 t^2 \log^3 k$. Given the condition on $|\ff(i)|$ for each $i\in[n]\setminus S$, at most $\frac {50t^2\log^3 k}{100(t\log k)^3}|\ff| = \frac 1{2t}|\ff|$ sets from $\ff$ intersect $I$.

  Let us next bound the total number of sets in ${[n]\choose k}$ intersecting at least $5\log k$ sets from $\G'$. For a fixed choice of $\ell\ge 5\log k$ sets, Claim~\ref{clanew2} asserts that the proportion of such sets in ${[n]\choose k}$ is at most $\Big(\frac {k^2}n\Big)^{\ell}$. At the same time, there are ${10t\log k\choose \ell}\le \Big(\frac {10 et\log k}{\ell}\Big)^{\ell}\le (2et)^{\ell}$ possible subsets of $\G'$, and so we can bound the number of all such sets in ${[n]\choose k}$ by
   \begin{equation}\label{eqnew3}
     {n\choose k}\cdot\sum_{\ell=5\log k}^{10t\log k}\Big(\frac {2et k^2}n\Big)^{\ell} \le {n\choose k}\cdot\begin{cases}
       2^{-5\log k}=k^{-5}, & \mbox{if } n\le k^{2.2} \\
       n^{-4} & \mbox{otherwise}.
     \end{cases}
   \end{equation}

%   \begin{equation}\label{eqnew3}
%     {n\choose k}\cdot\Big(\frac {2^{5/2} k^2}n\Big)^{40\log k} \le {n\choose k}\cdot\begin{cases}
%       2^{-40\log k}=k^{-40}, & \mbox{if } n\le k^3 \\
%       n^{-10} & \mbox{otherwise}.
%     \end{cases}
%   \end{equation}
  Note that we used the condition $n\ge 16t k^2$ in the first case. %(And this is the only place where we actually need to use it.)
  It is evident that in both cases the number of such sets is much smaller than ${n-4\choose k-4}$.

  Thus, the majority (at least $(1-\frac 1{2t})$-proportion, excluding those intersecting $I$) of sets from $\ff$ intersect at most $\frac 1{2t}$-proportion of the sets from $\G'$. Via simple double counting, it is clear that one of the sets from $\G'$ intersects at most a $\frac 1t$-fraction of sets from $\ff$.
\end{proof}
%{\bf Case $\tau (\ff) = t$. } The same proof with slightly different calculations allows to prove the following lemma:
%\begin{lem}\label{lemnew2}
%  In the assumptions of Theorem~\ref{thmnew1}, assume that there exists $S$ such that $|\ff(\bar S)|\ge {n-4\choose k-4}$ and, moreover, for any $i\in [n]\setminus S$ we have $|\ff(i)|\le \frac {|\ff|}{100(t\log k)^3}$. Then there is a set $F\in \ff(\bar S)$ that is disjoint with more than $\frac {(t-1)|\ff|}t$ sets from $\ff$.
%\end{lem}

\begin{cla}\label{clanew3} In the assumptions of Theorem~\ref{thmnew1}, there are fewer than $(10t\log k)^3$ elements $i\in [n]$ such that $|\ff(i)|\ge \frac{|\ff|}{100 (t\log k)^3}$.
\end{cla}
\begin{proof} Assume that there is a set $S$ of such elements of size $(10t\log k)^3$. By the inclusion-exclusion principle,
$$|\ff|\ge \sum_{i\in S}|\ff(i)|- \sum_{i\ne j\in S}|\ff(\{i,j\}|\ge 10|\ff|- |S|^2 {n-2\choose k-2}.$$
Note that $|S|^2 = (10t\log k)^6$ is less than $k$, on the other hand, applying Claim~\ref{clanew2} to $\D$ we infer $|\ff|>|\D|>\frac k2{n-2\choose k-2}$. This shows that the right hand side of the displayed inequality is more than $8|\ff|+\big(2|\ff|-k{n-2\choose k-2}\big)>8|\ff|$, giving the contradiction $|\ff|>8|\ff|$.
%since $|\ff|\ge |\D|\ge \frac k2{n-2\choose k-2}$ (this estimate on $|\D|$ follows from and $(10 t\log k)^6 \le k$.
\end{proof}

Combining Claim~\ref{clanew3} and Lemma~\ref{lemnew2}, we immediately get the following corollary.

\begin{cor}\label{cornew1}
  If $\ff$ satisfies the requirements of Theorem~\ref{thmnew1} and minimizes $d(\ff)$ for fixed $|\ff|$ then there exists a set $S$ of size at most $(10t\log k)^3$ such that $|\ff(\bar S)|\le {n-4\choose k-4}$.
\end{cor}

\begin{lem}\label{lemnew3}
  If $\ff, S$ are as in Corollary~\ref{cornew1} then $|S|\le t$.
\end{lem}
\begin{proof}
  Indeed, assume that $|S|>t$ and take $i\in S$ such that \begin{equation}\label{eqnew4} %\min\big\{\frac 2{2t+1}, \frac 5{|S|}\big\}
  \frac 3{3t+2}|\ff|\ge |\ff(\bar S')|\ge {n-2\choose k-2},\end{equation} where $S':= S\setminus \{i\}$. It is easy to see that the right inequality is satisfied for any $i\in S$, while the left one is satisfied for $i$ such that $|\ff(i)|$ is minimal over $i\in S$, provided $|S|\ge 3$. (Both inequalities are obtained using inclusion-exclusion principle in a similar way as used in the proof of Claim~\ref{clanew3}.)

  Next, we apply an argument very similar to the one given in the proof of Lemma~\ref{lemnew2}. We find a family $\G'\subset \ff(\bar S')$ of $10 t^2 \log k$ sets with pairwise intersections of size at most $\log k$. The set $I = \bigcup_{A\ne B\in \G'} A\cap B$ satisfies $|I|\le 50 t^4 \log^3 k$. %with pairwise intersections of size at most $10^4\log^3 k$. The set $I = \bigcup_{A\ne B\in \G'} A\cap B$ has size at most $10^4 \log^3 k$.
  The number of sets from $\ff\setminus \ff(\bar S')$ intersecting $I$ is at most
  \begin{equation}\label{eqnew5} |I||S'|{n-2\choose k-2}\le 50 t^4 \log^3 k\cdot (10t\log k)^3{n-2\choose k-2}\le \frac{1}{20 t}|\ff|,
  \end{equation}
  where the last inequality holds provided $k$ is large enough.

  Next, we estimate the number of sets intersecting at least $5\log k$ sets from $\G'$. We do virtually the same calculation as in Lemma~\ref{lemnew2} %, except we bound the number of $20\log k$-element subsets in $\G'$ by $2^{H(0.2)\cdot 100 \log k}< 2^{80 \log k}$
  and get that there are at most
    \begin{equation*}
     {n\choose k}\cdot\sum_{\ell=5\log k}^{10t^2\log k}\Big(\frac {2e t^2 k^2}n\Big)^{\ell} \le {n\choose k}\cdot\begin{cases}
       2^{-5\log k}=k^{-5}, & \mbox{if } n\le t^2k^3 \\
       n^{-5} & \mbox{otherwise}
     \end{cases}
\end{equation*}
sets in ${[n]\choose k}$ with this property. This number is at most ${n-2\choose k-2}\le \frac 1{20t}|\ff|$ in any case.

We conclude that at least $1-\frac 3{3t+2}-\frac 1{10t}\ge 1-\frac 2{2t+1}$ of the sets from $\ff$ intersect at most $\frac 1{2t^2}$-fraction of sets from $\G'$. (Note that we had to exclude $\ff(i)$ itself, which made the biggest contribution to the complement.) Since $\big(1-\frac 2{2t+1}\big)\big(1-\frac 1{2t^2})>1-\frac 1t$, we again get by  simple double-counting that there is a set in $\G'$ that intersects fewer than $\frac 1t$-fraction of sets from $\ff$, which is a contradiction.
\end{proof}
Finally, it is not difficult to see that the condition $(10 t\log k)^7\le k$ is sufficient for the argument to go through. (And is only used in Claim~\ref{clanew3} and Lemma~\ref{lemnew3}.)

\section{Proof of Theorem~\ref{thmnew2}}
Let us first analyze the two possible cases: $|S| = 1$ and $|S| = 2$. Assume that $S = \{1\}$ in the first case and $S = \{1,2\}$ in the second case. We assume that $\ff$ minimizes $d(\ff)$ for fixed $|\ff|$ in either case.

  Assume that  $|S| = 1$ and  $|\ff|\le {n-1\choose k-1}+1$. Then  $\ff(\bar 1)$ contains at least one set $U$, while $|\ff(1)|\ge |\ff|-{n-4\choose k-4}$ due to Theorem~\ref{thmnew1}. At the same time, the number of sets from $\ff(1)$ intersecting $U$ is at most $|\mathcal H(1,U)|$. Thus,
\begin{equation}\label{eqs=1}
  |\ff|-|\mathcal H(1,U)|-{n-4\choose k-4}\le d(\ff).
\end{equation}
We can bound $d(|\ff|,n,k)$ from above using the following construction: %The upper bound is obtained by considering
take $\mathcal H(1,U)$ together with some other sets containing $1$ (so that we have $|\ff|$ sets in total). This gives
\begin{equation}\label{eqs=12}
  d(\ff)\le |\ff|-|\mathcal H(1,U)|.
\end{equation}

  Next, assume that $|S| = 2$. To analyze $d(\ff)$, we shall apply the following proposition. %to the Kneser graph $KG(n-2,k-1)$.
  \begin{prop}{\cite[Corollary 9.2.5]{AS}}\label{propal} Let $G = (V,E)$ be a $D$-regular $N$-vertex graph. Let $\lambda = \lambda (G)$ be the second largest absolute value of an eigenvalue of $G$. Then for any subsets $B,C\subset V$, where $|B| = bN$, $|C| = cN$, we have
  $$|e(B,C)-cbDN|\le \lambda \sqrt {bc}N.$$
  \end{prop}

  Let $\ff_1,\ff_2$ be defined as follows: $\ff_i:=\{U\setminus \{i\}: U\in \ff, U\cap [2] = i\}$.
Both $\ff_1,\ff_2$ can be seen as subsets of $KG(n-2,k-1)$ (on vertex set $[3,n]$). We note that $KG(n-2,k-1)$ is regular of degree ${n-k-1\choose k-1}$, and $\lambda= \lambda(KG(n-2,k-1)) = {n-k-2\choose k-2}$. Let $\delta_i$ be the average degree of a vertex in $\ff_i$. Applied to our situation, the proposition above implies the following.
\begin{prop}
Assume that $\{i,j\} = \{1,2\}$. Then
\begin{equation}\label{eq9}
  \delta_i \ge \frac{|\ff_j|{n-k-1\choose k-1}}{{n-2\choose k-1}}- {n-k-2\choose k-2}\sqrt{\frac{|\ff_j|}{|\ff_i|}}.
\end{equation}
Moreover,
\begin{equation}\label{eq17}d(\ff)\ge \frac{\frac{|\ff_1|+|\ff_2|}2{n-k-1\choose k-1}}{{n-2\choose k-1}} - {n-k-2\choose k-2}.\end{equation}
\end{prop}
\begin{proof} The first part is just an application of Proposition~\ref{propal}. Next, assume that $|\ff_1|\ge |\ff_2|.$
It is not difficult to see that, for fixed $|\ff_1|+|\ff_2|$, the bound on $\delta_2$ is the smallest when $|\ff_1|=|\ff_2|$. Indeed, recall that $|\ff|\ge |\D|$, and thus $|\ff_1|,|\ff_2|\ge \frac 13 |\D|\ge \sqrt k{n-k-2\choose k-2}$ by Theorem~\ref{thmnew1}. The first term on the right hand side of \eqref{eq9} is essentially $C|\ff_1|$, where $C>\frac 12$, while the absolute value of the second term grows slower than $c|\ff_1|$ with $c\le {n-k-2\choose k-2}/|\ff_2|\le k^{-0.5}$. Thus, we get \eqref{eq9}.
\end{proof}
%Given that $|\ff_1|+|\ff_2| = |\ff|-{n-2\choose k-2}$, we get the following bound:

We note that \begin{equation}\label{eq667} {n-2\choose k-2}\cdot {n-k-1\choose k-1}/{n-2\choose k-1} = \frac {k-1}{n-k}{n-k-1\choose k-1}<{n-k-2\choose k-2}.\end{equation}
Knowing that $|\ff(S)|\le {n-2\choose k-2}$ and $|\ff(\bar S)|\le {n-4\choose k-4}$, from \eqref{eq17} and \eqref{eq667} we derive that
\begin{multline}\label{eqs=2}d(\ff)\ge \frac 12\frac{(|\ff|-|\ff(\bar S)|-|\ff(S)|){n-k-1\choose k-1}}{{n-2\choose k-1}} - {n-k-2\choose k-2}\\ \ge \frac 12\frac{|\ff|{n-k-1\choose k-1}}{{n-2\choose k-1}} - \frac 32{n-k-2\choose k-2}-\frac 12{n-4\choose k-4}\ge 0.4|\ff|.\end{multline}

We go on to the proof of different parts of Theorem~\ref{thmnew2}. \\

1. We first note that $|S| = 1$. Indeed, \eqref{eqs=1} implies that $d(\mathcal E_i)<\frac 13|\mathcal E_i|$ in this range, while \eqref{eqs=2} gives $d(\ff)\ge 0.4|\ff|$.

Next, assume that there are at least $k+i+1$ elements $y\in [2,n]$ such that $|\ff(\{1,y\})|\ge {n-2\choose k-2}- {n-4\choose k-4}$. Then $|\ff|\ge {n-1\choose k-1}-{n-k-i-2\choose k-1}-(k+i+1){n-4\choose k-4}>{n-1\choose k-1}-{n-k-i-1\choose k-1}+{n-4\choose k-4}$, a contradiction with the assumption on the size of $\ff$. Consequently, there are at most $k+i$ such elements.

In particular, if $|\bigcup_{U\in \ff(\bar 1)} U|>k+i$, then there is a set $U\in \ff(\bar 1)$ such that one of its elements, say, $y$, is contained in  less than ${n-2\choose k-2}- {n-4\choose k-4}$ sets containing $1$. Then the degree of this set is at least $|\ff|-{n-1\choose k-1}+{n-k-1\choose k-1}+{n-4\choose k-4}-|\ff(\bar 1)|\ge |\ff|-{n-1\choose k-1}+{n-k-1\choose k-1}$. This is more than the maximum degree in a family of size $|\ff|$ having just one set $U$ not containing $1$ and containing all sets containing $1$ and passing through $U$.

Finally, assume that $|\ff| = |\mathcal E_i|$. If $|\ff(\bar 1)|<{k+i\choose k}$ then, for any $U\in \ff(\bar 1)$, the number of sets not intersecting it is at least $|\mathcal E_i|- {n-1\choose k-1}+{n-k-1\choose k-1}-{k+i\choose k}+1$, which is bigger than the right hand side in \eqref{eqdei}. Thus, $|\ff(\bar 1)| = {k+i\choose k}$ and, by the first part, must have the form ${Y\choose k}$ for some $Y\subset [2,n]$ of size $k+i$. % On the other hand, if $\ff(\bar 1)$ has structure different from ${X\choose k}$ for some $X$ of size $k+i$, then we can argue as above to show that it cannot minimize maximum degree.
\\

2, 3. In the cases 2, 3 of Theorem~\ref{thmnew2}, we have $|S| = 2$. Indeed, if $|S| = 1$ then \eqref{eqs=1} implies that $d(\ff)\ge \frac 12 |\ff|$. Next, if we have $|\ff|>4|\D|$ and $\ff(\bar \emptyset) \ne\emptyset$, then, for any set $U\in \ff(\bar S)$, there are at most $2|\D|-{n-3\choose k-3}$ sets from $\ff\setminus \ff(\bar S)$ intersecting it, and thus $d(\ff)\ge |\ff|-2|\D|+{n-3\choose k-3}-{n-4\choose k-4}>\frac 12 |\ff|$.

\section{Proof of Theorem~\ref{thm3}}

  Take a family $\ff$ satisfying the restrictions of the theorem and such that  $d(\ff) = d(|\ff|,n,k)$. The conclusion of Theorem~\ref{thmnew2} holds, i.e., we must have two elements, say,  $1,2$ such that $\ff(\bar{[2]}) = \emptyset$. We may w.l.o.g. assume that $\ff([2]) = {[3,n]\choose k-2}$. Next, for $i=1,2$ consider the families $\ff_i:=\{F\setminus \{i\}: F\in \ff, F\cap [2] = \{i\}\}.$ Inequality \eqref{eq17} is valid in our situation, which, together with $|\ff_1|+|\ff_2| = |\ff|-{n-2\choose k-2}$, concludes the proof of the first part of \eqref{eq10}.

  %These families are $(k-1)$-uniform.

 % We note that $$\big||\ff_1|-|\ff_2|\big|\le 2|\D|.$$ Indeed, if this is false and $|\ff_1|>|\ff_2|$, then any set $F\in \ff_2$ is disjoint from at least $|\ff_1|-|\D|>(|\ff_1|+|\ff_2|)/2$ sets, and then a family $\ff'$ with equally sized $\ff'_1$ and $\ff'_2$ has smaller maximum degree.

%Given that $|\ff_1|+|\ff_2| = |\ff|-{n-2\choose k-2}$, we get the following bound:

%$$d(\ff)\ge \frac{\frac{|\ff|-{n-2\choose k-2}}2{n-k-1\choose k-1}}{{n-2\choose k-1}} - {n-k-2\choose k-2}.$$

To get to the second form of the bound in \eqref{eq10}, we use the same calculations as in \eqref{eqs=2} together with the fact that $|\ff(\bar S)| =0$, and get % and thus
$$
d(\ff)\ge \frac{|\ff|{n-k-1\choose k-1}}{2{n-2\choose k-1}} - \frac 32{n-k-2\choose k-2}.$$ Finally, we note that $\frac{{n-k-1\choose k-1}}{{n-2\choose k-1}} \ge \Big(\frac{n-2k}{n-k}\Big)^{k-1} = (1-\frac k{n-k})^{k-1}\ge (1-\frac {k+1}{n})^{k-1}\ge 1-\frac {(k+1)(k-1)}n\ge 1-\frac {k^2}n$. We have used the fact that $n\ge k^2+k$ in the second inequality.

\section{Optimality of Theorem~\ref{thm3}}\label{secopt}
Fix some $s>1000 k$ and $n>10k^3$ and consider the following family $\G$: $$\G:=\big\{A\in {[n]\choose k}: \{1,2\}\subset A\big\}\cup\big\{A\in {[n]\choose k}: |A\cap \{1,2\}|=1, |A\cap [3,s+2]|=1\big\}.$$
It is not difficult to see that $|\G| = {n-2\choose k-2}+2s{n-s-2\choose k-2}$. Substituting this into the bound \eqref{eq10}, we get that
\begin{align}d(\G)\ge& \frac{s {n-s-2\choose k-2}{n-k-1\choose k-1}}{{n-2\choose k-1}} - {n-k-2\choose k-2} = \frac{(s-1) {n-s-2\choose k-2}{n-k-1\choose k-1}}{{n-2\choose k-1}}-\Theta\Big(s{n-3\choose k-3}\Big)\notag\\
\label{eq11}\ge & (s-1) {n-s-k\choose k-2}-\Theta\Big(s{n-3\choose k-3}\Big),\end{align}
where $\Theta$ stands for a constant, independent of $s,k,n$. In the last transition, we have used the result of the following calculation:\begin{small}
$$\frac{{n-s-2\choose k-2}{n-k-1\choose k-1}}{{n-2\choose k-1}} = \prod_{i=1}^{k-1}\frac{{n-k-i}}{{n-1-i}}{n-s-2\choose k-2} = \prod_{i=1}^{k-1}\frac{{n-k-i}}{{n-1-i}} \prod_{j=0}^{k-3}\frac{{n-s-2-j}}{{n-s-k-j}} {n-s-k\choose k-2} $$
$$\ge  \prod_{i=1}^{k-1}\frac{{n-k-i}}{{n-1-i}} \prod_{j=0}^{k-4}\frac{{n-s-2-j}}{{n-s-1-k-j}} {n-s-k\choose k-2}\ge \prod_{i=1}^{2}\frac{{n-k-i}}{{n-1-i}}  {n-s-k\choose k-2}\ge \Big(1-\frac {2k}n\Big){n-s-k\choose k-2}, $$\end{small}
which implies
$$\frac{{n-s-2\choose k-2}{n-k-1\choose k-1}}{{n-2\choose k-1}}-{n-s-k\choose k-2}\ge -\frac{2k}n{n-s-k\choose k-2}\ge -4{n-3\choose k-3},$$
for $k\ge 4$.

At the same time, it is easy to see that each set from $\G$ that intersects $[2]$ in $1$ element has the same degree, namely $(s-1){n-s-k\choose k-2}.$ This is the maximum degree of $KG(\G)$. Compare this with \eqref{eq11}. Note that $s{n-3\choose k-3}\ll {n-s-k\choose k-2}$ for $s\ll n/k$.
\\

{\sc Acknowledgements. } The research of the first author was partially supported by the National Research, Development and Innovation Office - NKFIH under the grant K 132696. The authors acknowledge the financial support from the Ministry of Education and Science of the Russian Federation in the framework of MegaGrant no 075-15-2019-1926.

\end{document}